\documentclass[12pt]{amsart}

\usepackage{graphics, amsmath, amsfonts, amssymb, amsthm, amscd, mathrsfs, color}

\input xy
\xyoption{all}

\newtheorem{theorem}{Theorem}

\newtheorem{lemma}[theorem]{Lemma}
\newtheorem{corollary}[theorem]{Corollary}

\theoremstyle{definition}

\newtheorem{remark}[theorem]{Remark}

\newcommand{\C}{{\mathbb C}}
\newcommand{\R}{{\mathbb R}}
\renewcommand{\O}{{\mathscr O}}
\newcommand{\Z}{{\mathbb Z}}
\newcommand{\B}{{\mathbb B}}

\title{Representing de Rham cohomology classes \\ on an open Riemann surface by holomorphic forms}

\author{Antonio Alarc\'on and Finnur L\'arusson}

\address{Antonio Alarc\'on, Departamento de Geometr\'ia y Topolog\'ia e Instituto de Mate\-m\'aticas (IEMath-GR), Universidad de Granada, Campus de Fuentenueva s/n, E-18071 Granada, Spain}
\email{alarcon@ugr.es}

\address{Finnur L\'arusson, School of Mathematical Sciences, University of Adelaide, Adelaide SA 5005, Australia}
\email{finnur.larusson@adelaide.edu.au}

\subjclass[2010]{Primary 30F30.  Secondary 30F99, 32H02, 53A10, 55P10.}

\keywords{Riemann surface, de Rham cohomology, minimal surface, holomorphic null curve, Serre fibration, convex integration, period-dominating spray}

\date{1 April 2017}

\thanks{A. Alarc\'on is supported by the Ram\'on y Cajal program of the Spanish Ministry of Economy and Competitiveness and supported in part by the MINECO/FEDER grant no.\ MTM2014-52368-P, Spain.  F.~L\'arusson is supported by Australian Research Council grant DP150103442.  The work on this paper was done while A. Alarc\'on visited the University of Adelaide in March 2017.  He would like to thank the university for the hospitality.}

\begin{document}

\begin{abstract}    
Let $X$ be a connected open Riemann surface.  Let $Y$ be an Oka domain in the smooth locus of an analytic subvariety of $\C^n$, $n\geq 1$, such that the convex hull of $Y$ is all of $\C^n$.  Let $\O_*(X, Y)$ be the space of nondegenerate holomorphic maps $X\to Y$.  Take a holomorphic $1$-form $\theta$ on $X$, not identically zero, and let $\pi:\O_*(X,Y) \to H^1(X,\C^n)$ send a map $g$ to the cohomology class of $g\theta$.  Our main theorem states that $\pi$ is a Serre fibration.  This result subsumes the 1971 theorem of Kusunoki and Sainouchi that both the periods and the divisor of a holomorphic form on $X$ can be prescribed arbitrarily.  It also subsumes two parametric h-principles in minimal surface theory proved by Forstneri\v c and L\'arusson in 2016.
\end{abstract}

\maketitle

\noindent
We start by recalling three theorems from the early decades of modern Riemann surface theory.  Behnke and Stein proved that the periods of a holomorphic form on an open Riemann surface $X$ can be prescribed arbitrarily \cite[Satz 10]{Behnke-Stein-1947} (see also \cite[Theorem 28.6]{Forster-1981}).  In other words, every class in the cohomology group $H^1(X,\C)$ contains a holomorphic form.  Gunning and Narasimhan showed that the zero class contains a holomorphic form with no zeros \cite{Gunning-Narasimhan-1967}.  In other words, there is a holomorphic immersion $X\to\C$.  Kusunoki and Sainouchi generalised these two theorems and proved that both the periods and the divisor of a holomorphic form on $X$ can be prescribed arbitrarily \cite[Theorem 1]{Kusunoki-Sainouchi-1971}.

Our main result subsumes the theorem of Kusunoki and Sainouchi as a very special case.  It also subsumes different and much more recent results from the theory of minimal surfaces, which we shall now describe.  Let $\mathfrak M_*(X,\R^n)$ denote the space (with the compact-open topology) of conformal minimal immersions $X\to\R^n$, $n\geq 3$, that are nonflat in the sense that the image of $X$ is not contained in any affine $2$-plane in $\R^n$.  Some such immersions are obtained as the real parts of holomorphic null curves, that is, holomorphic immersions $X\to\C^n$ directed by the null quadric $\mathfrak A=\{(z_1,\ldots,z_n)\in\C^n:z_1^2+\cdots+z_n^2=0\}$.  We denote by $\mathfrak N_*(X,\C^n)$ the space of holomorphic null curves that are nonflat, meaning that the image of $X$ is not contained in any affine complex line in $\C^n$.  

Forstneri\v c and L\'arusson determined the weak homotopy type of $\mathfrak M_*(X,\R^n)$ and $\mathfrak N_*(X,\C^n)$ \cite{Forstneric-Larusson-CAG}.  The two spaces turn out to have the same weak homotopy type, namely the weak homotopy type of the space of continuous maps $X\to\mathfrak A_*=\mathfrak A\setminus\{0\}$, which we understand.  Key ingredients in the proof are two parametric h-principles \cite[Theorems 4.1 and 5.3]{Forstneric-Larusson-CAG} that imply that the maps
\[ p:\mathfrak M_*(X,\R^n) \to \O_*(X,\mathfrak A_*), \quad u\mapsto \partial u/\theta, \]
\[ q:\mathfrak N_*(X,\C^n) \to \O_*(X,\mathfrak A_*), \quad F\mapsto \partial F/\theta, \]
are weak homotopy equivalences.  Here, $\theta$ is a chosen holomorphic $1$-form on $X$ without zeroes and $\O_*(X,\mathfrak A_*)$ is the space of holomorphic maps $f:X\to \mathfrak A_*$ that are nondegenerate in the sense that the tangent spaces $T_{f(x)}\mathfrak A\subset T_{f(x)}\C^n\cong\C^n$, as $x$ ranges through $X$, span $\C^n$.  These parametric h-principles are proved using Gromov's method of convex integration and period-dominating sprays, a tool from Oka theory introduced in \cite{Alarcon-Forstneric-2014}.  Our main theorem implies that both $p$ and $q$ are weak homotopy equivalences.

We now proceed to the statement of our main result.  Let $X$ be an open Riemann surface, always assumed connected.  We view the cohomology group $H^1(X,\C)$ as the de Rham group of holomorphic $1$-forms modulo exact forms, with the quotient topology induced from the compact-open topology.  

Let $Y$ be an Oka domain in the smooth locus of an analytic subvariety of $\C^n$, $n\geq 1$, such that the convex hull of $Y$ is all of $\C^n$.  For example, $Y$ could be $\mathfrak A_*$.  Also, $Y$ could be a smooth algebraic subvariety of $\C^n$ that is Oka and not contained in any hyperplane; then the convex hull of $Y$ is $\C^n$ \cite[Lemma 3.1]{Alarcon-Forstneric-2014}.

Let $\O_*(X, Y)$ be the space of holomorphic maps $f:X\to Y$ that are nondegenerate in the sense that the tangent spaces $T_{f(x)} Y\subset T_{f(x)}\C^n\cong\C^n$, as $x$ ranges through $X$, span $\C^n$.  We note that $\O_*(X, Y)$ is not empty.  Indeed, since the convex hull of $Y$ is $\C^n$, $Y$ is not contained in any hyperplane, so the tangent spaces of $Y$ at finitely many points $a_1,\ldots,a_m$ in $Y$ span $\C^n$.  There is a continuous map $X\to Y$ with $a_1,\ldots,a_m$ in its image, and since $Y$ is Oka, the map can be deformed to a holomorphic map keeping it fixed at a preimage of each $a_j$.  This holomorphic map is then nondegenerate.

Take a holomorphic $1$-form $\theta$ on $X$, not identically zero, and let the continuous map
\[\pi:\O_*(X,Y) \to H^1(X,\C^n) \] 
send a map $g$ to the cohomology class of $g\theta$.  The map $\pi$ is our central object of interest.  Our main result states that $\pi$ is a Serre fibration.  Here it is, in a more explicit formulation (see Remark \ref{rem:equivalence}).

\begin{theorem}  \label{thm:main}
Let $X$ be an open Riemann surface, $\theta$ be a holomorphic $1$-form on $X$, not identically zero, and $Y$ be an Oka domain in the smooth locus of an analytic subvariety of $\C^n$, $n\geq 1$, such that the convex hull of $Y$ is all of $\C^n$.  Let $\O_*(X,Y)$ be the subspace of $\O(X,Y)$ of nondegenerate maps and $\pi:\O_*(X,Y) \to H^1(X,\C^n)$ be the projection sending a map $g$ to the cohomology class of $g\theta$.

Let $Q\subset P$ be compact subsets of $\R^m$ for some $m$.  Let $f:P\to\O_*(X,Y)$ and $\phi:P\to H^1(X,\C^n)$ be continuous maps such that $\pi\circ f=\phi$ on $Q$.  Then $f$ can be deformed, keeping $f\vert_Q$ fixed, to a continuous map $f^1:P\to\O_*(X,Y)$ with $\pi\circ f^1=\phi$ on $P$.
\end{theorem}

The spaces and maps in the theorem are shown in the following diagram.
\[ \xymatrix{
Q \ar[r] \ar@{^{(}->}[d]  &  \O_*(X,Y) \ar^\pi[d] \\
P \ar_{\!\!\!\!\!\!\!\!\!\!\!\!\!\!\!\!\phi}[r] \ar^f[ur] &  H^1(X,\C^n)
} \]
The square and the upper triangle commute, whereas the lower triangle need not commute.  The theorem states that $f$ can be deformed, keeping the upper triangle commuting, until the lower triangle commutes.

We postpone the proof of the theorem until the end of the paper.

\begin{corollary}  \label{cor:main}
Under the assumptions of Theorem \ref{thm:main}, 
\begin{enumerate}
\item[(a)]  $\pi:\O_*(X,Y) \to H^1(X,\C^n)$ is a Serre fibration,
\item[(b)]  $\pi$ is surjective,
\item[(c)]  for every $\alpha\in H^1(X,\C^n)$, the inclusion $\pi^{-1}(\alpha)\hookrightarrow \O_*(X,Y)$ is a weak homotopy equivalence.
\end{enumerate}
\end{corollary}

In the following, we will simply refer to a Serre fibration as a fibration.  Also, a weak homotopy equivalence will be called a weak equivalence.

\begin{proof}
(a)  We need to show that $\pi$ has the right lifting property with respect to every inclusion $\iota:I^k\times\{0\}\hookrightarrow I^{k+1}$, $k\geq 0$, where $I=[0,1]$.  Let $P=I^{k+1}$ and $Q=I^k\times\{0\}$.  Let $g:Q \to \O_*(X,Y)$ and $\phi:P\to H^1(X,\C^n)$ be continuous maps such that $\pi\circ g=\phi\circ\iota$.  Extend $g$ to a continuous map $f:P\to\O_*(X,Y)$ that is constant in the $(k+1)^\textrm{st}$ variable.  Then $\pi\circ f=\phi$ on $Q$, so by Theorem \ref{thm:main}, $f$ can be deformed, keeping $f\vert_Q$ fixed, to a continuous map $f^1:P\to\O_*(X,Y)$ with $\pi\circ f^1=\phi$.  Then also $f^1\circ\iota=g$, and $f^1$ demonstrates the right lifting property.

(b) follows from (a) since $H^1(X,\C^n)$ is path connected.

(c)  The inclusion $\pi^{-1}(\alpha)\hookrightarrow \O_*(X,Y)$ is a weak equivalence because it is the pullback by the fibration $\pi$ of the inclusion $\{\alpha\}\hookrightarrow H^1(X,\C^n)$, which is a weak equivalence since $H^1(X,\C^n)$ is contractible.
\end{proof}

\begin{remark}  \label{rem:equivalence}
If we restrict the conclusion of Theorem \ref{thm:main} to finite polyhedra, then it is \textit{equivalent} to $\pi$ being a fibration.  The equivalence relies on $H^1(X,\C^n)$ being contractible.  (Fibrations are characterised by the right lifting property with respect to several different classes of inclusions, so finite polyhedra are not the only possibility here.)

To demonstrate the equivalence, suppose that $\pi$ is a fibration.  Let $Q$ be a subpolyhedron of a finite polyhedron $P$, and let $g:Q\to \O_*(X,Y)$ and $\phi:P\to H^1(X,\C^n)$ be continuous maps such that $\pi\circ g=\phi\vert_Q$.  Consider the following commuting diagram.
\[ \xymatrix{
\{f\in\mathscr C(P,\O_*):f\vert_Q=g,\, \pi\circ f=\phi\} \ar@{^{(}->}[d] \ar@{^{(}->}[r] & \{f\in\mathscr C(P,\O_*):f\vert_Q=g\} \ar@{^{(}->}[d] \\
\mathscr C(P, \O_*) \ar_{(\pi_*,\iota^*)}[d] \ar@{=}[r] & \mathscr C(P, \O_*) \ar^{\iota^*}[d] \\
\mathscr C(P,H^1)\times_{\mathscr C(Q,H^1)} \mathscr C(Q,\O_*) \ar[r] & \mathscr C(Q,\O_*)
} \]
Here, $\iota^*$ is the restriction map induced by the inclusion $\iota:Q\hookrightarrow P$, the left-hand top space is the fibre of $(\pi_*,\iota^*)$ over $(\phi,g)$, the right-hand top space is the fibre of $\iota^*$ over $g$, and we have abbreviated $\O_*(X,Y)$ as $\O_*$ and $H^1(X,\C^n)$ as $H^1$.

Since $\iota$ is a cofibration (in the strong sense that goes with Serre fibrations), $\iota^*$ is a fibration.  Since also $\pi$ is a fibration by assumption, $(\pi_*,\iota^*)$ is a fibration by Quillen's axiom SM7 \cite[II.3.1]{Goerss-Jardine-1999}.  The bottom horizontal map is a weak equivalence, because it is the pullback of a weak equivalence by a fibration, as shown by the following pullback square.
\[ \xymatrix{
\mathscr C(P,H^1)\times_{\mathscr C(Q,H^1)} \mathscr C(Q,\O_*) \ar[r] \ar[d] & \mathscr C(P,H^1) \ar[d] \\
\mathscr C(Q,\O_*) \ar_{\pi_*}[r] & \mathscr C(Q,H^1)
} \]
Here, the restriction map $\mathscr C(P,H^1) \to \mathscr C(Q,H^1)$ is a weak equivalence because $H^1$ is contractible, and $\pi_*$ is a fibration because $\pi$ is.

Finally, the coglueing lemma \cite[Lemma II.8.10]{Goerss-Jardine-1999} implies that the top inclusion 
\[ \xymatrix{
\{f\in\mathscr C(P,\O_*):f\vert_Q=g, \, \pi\circ f=\phi\} \ar@{^{(}->}[r] & \{f\in\mathscr C(P,\O_*):f\vert_Q=g\} 
} \]
is a weak equivalence (whereas all we wanted to show was that this inclusion induces a surjection of path components).  
\end{remark}

The next result is our generalisation of \cite[Theorem 1]{Kusunoki-Sainouchi-1971} for holomorphic forms.

\begin{corollary}  \label{cor:generalisation-of-K-S}
{\rm (a)}  Let $X$ be an open Riemann surface.  Let $\theta$ be a holomorphic $1$-form on $X$ that does not vanish everywhere.  Then the map $\O(X,\C^*) \to H^1(X,\C)$, sending a function $f$ to the cohomology class of $f\theta$, is a Serre fibration.

{\rm (b)}  Let $D$ be an effective divisor on $X$.  Every cohomology class in $H^1(X,\C)$ contains a holomorphic $1$-form with divisor $D$.
\end{corollary}

\begin{proof}
(a)  Apply Corollary \ref{cor:main}(a) with $Y=\C^*\subset\C$ and note that every holomorphic function on $X$ with no zeros is nondegenerate.

(b)  Take $Y=\C^*\subset\C$, let $\theta$ be a holomorphic $1$-form on $X$ with divisor $D$, and apply Corollary \ref{cor:main}(b).  (The existence of such a form does not rely on \cite{Kusunoki-Sainouchi-1971}.  It is an immediate consequence of \cite[Satz 4]{Florack-1948}.)
\end{proof}

Now we derive the results from \cite{Forstneric-Larusson-CAG} mentioned above.

\begin{corollary}  \label{cor:minimal}
Let $X$ be an open Riemann surface, let $\theta$ be a holomorphic $1$-form on $X$ with no zeroes, and let $n\geq 3$.  The maps
\[ p:\mathfrak M_*(X,\R^n) \to \O_*(X,\mathfrak A_*), \quad u\mapsto \partial u/\theta, \]
\[ q:\mathfrak N_*(X,\C^n) \to \O_*(X,\mathfrak A_*), \quad F\mapsto \partial F/\theta, \]
are weak equivalences.
\end{corollary}

\begin{proof}
We already noted that $\mathfrak A_*$ satisfies the assumptions on $Y$ in Theorem \ref{thm:main}.  As before, let $\pi:\O_*(X,\mathfrak A_*) \to H^1(X,\C^n)$ send a map $g$ to the cohomology class of $g\theta$.  Choose a base point $a$ in $X$.  First, $q$ is the composition of the inclusion $\pi^{-1}(0)\hookrightarrow \O_*(X,\mathfrak A_*)$, which is a weak equivalence by Corollary \ref{cor:main}(c), and the map $\mathfrak N_*(X,\C^n)\to \pi^{-1}(0)$ induced by $q$, which in turn is the composition of the homeomorphism $\mathfrak N_*(X,\C^n)\to \pi^{-1}(0)\times \C$, $F\mapsto (\partial F/\theta, F(a))$, and the projection $\pi^{-1}(0)\times \C^n \to \pi^{-1}(0)$, which clearly is a weak equivalence.

As for $p$, it is the composition of the inclusion $\pi^{-1}(H^1(X,i\R^n))\hookrightarrow \O_*(X, \mathfrak A_*)$, which is a weak equivalence because it is the pullback of the weak equivalence $H^1(X,i\R^n) \hookrightarrow H^1(X,\C^n)$ by the fibration $\pi$, and the map $\mathfrak M_*(X,\R^n)\to \pi^{-1}(H^1(X,i\R^n))$ induced by $p$, which in turn is the composition of the homeomorphism 
\[ \mathfrak M_*(X,\R^n)\to \pi^{-1}(H^1(X,i\R^n))\times \R^n, \quad u\mapsto (\partial u/\theta, u(a)), \] 
and the projection $\pi^{-1}(H^1(X,i\R^n))\times \R^n \to \pi^{-1}(H^1(X,i\R^n))$, which clearly is a weak equivalence.
\end{proof}

Let us sketch a simplified approach to the main result of \cite{Forstneric-Larusson-CAG}.  Consider the following commuting diagram \cite[(6.1)]{Forstneric-Larusson-CAG}.
\[ \xymatrix{
\mathfrak N_*(X,\C^n)  \ar^q[r] \ar_{\mathrm{Re}}[d]  &  \O_*(X,\mathfrak A_*)  \ar@{^{(}->}^i[r] &  \O(X,\mathfrak A_*) \ar@{^{(}->}^j[r]  &   \mathscr C (X,\mathfrak A_*)  \\ 
\operatorname{Re}\mathfrak N_*(X,\C^n)   \ar@{^{(}->}[r]     &  \mathfrak M_*(X,\R^n)  \ar_{2p}[u] & &
} \]
(The factor of $2$ is needed to make the square commute.)  We have seen how $\pi$ being a fibration easily implies that $2p$ and $q$ are weak equivalences.  As explained in \cite{Forstneric-Larusson-CAG}, the real-part map on the left is a homotopy equivalence by continuity of the Hilbert transform.  It is then immediate that the inclusion $\operatorname{Re}\mathfrak N_*(X,\C^n) \hookrightarrow \mathfrak M_*(X,\R^n)$ is a weak equivalence: this is \cite[Theorem 1.1]{Forstneric-Larusson-CAG}.  A general position argument shows that $i$ is a weak equivalence \cite[Theorem 5.4]{Forstneric-Larusson-CAG}.  Finally, $j$ is a weak equivalence by the standard parametric Oka principle.  Thus, the six spaces in the diagram all have the same weak homotopy type.

We conclude the paper by proving our main result.  The proof closely follows the proofs of \cite[Theorems 4.1 and 5.3]{Forstneric-Larusson-CAG}.

\begin{proof}[Proof of Theorem \ref{thm:main}]
Choose a smooth strongly subharmonic Morse exhaustion function $\rho\colon X\to\R$ and an increasing sequence $c_1<c_2<c_3<\cdots$ of regular values of $\rho$ such that each interval $[c_j,c_{j+1}]$ contains at most one critical value of $\rho$. Set
\[ D_j = \{x\in X\colon \rho(x)<c_j\},\quad j\geq 1, \]
and note that
\begin{equation}\label{eq:cupDj}
	D_1\Subset D_2\Subset D_3\Subset\cdots\Subset \bigcup_{j\geq 1} D_j=X
\end{equation}
is an exhaustion of $X$ by smoothly bounded, relatively compact domains whose closures are $\O(X)$-convex.  We may assume that $D_1$ is simply connected.  Denote by $A$ the analytic subvariety of $\C^n$ in which $Y$ is a domain.

We claim that to prove the theorem it suffices to construct a sequence of homotopies of nondegenerate holomorphic maps 
\[ f_{p,j}^t\colon \overline D_j\to Y,\quad (p,t)\in P\times[0,1], \]
and a sequence of numbers $\epsilon_j>0$, $j\geq 1$, such that the following properties are satisfied for all $j$.
\begin{itemize}
\item[\rm (1$_j$)] $f_{p,j}^t=f(p)|_{\overline D_j}$ for all $(p,t)\in (P\times\{0\})\cup(Q\times[0,1])$.
\item[\rm (2$_j$)] $\|f_{p,j}^t-f_{p,j-1}^t\|_{\overline D_{j-1}}<\epsilon_j$ for all $(p,t)\in P\times[0,1]$.
\item[\rm (3$_j$)] $\int_\gamma f_{p,j}^t\theta=\int_\gamma f_{p,j-1}^t\theta$ for all loops $\gamma$ in $D_{j-1}$ and all $(p,t)\in P\times[0,1]$.
\item[\rm (4$_j$)] $\int_\gamma f_{p,j}^1\theta=\int_\gamma \phi(p)$ for all loops $\gamma$ in $D_j$ and all $p\in P$.
\item[\rm (5$_j$)] $2\epsilon_j<\epsilon_{j-1}$.
\item[\rm (6$_j$)] If $h\colon X\to A\subset\C^n$ is a holomorphic map such that $\|h-f_{p,j}^t\|_{\overline D_j}<2\epsilon_j$ for some $(p,t)\in P\times[0,1]$, then $h\vert_{\overline D_j}$ takes its values in $Y$ and is nondegenerate.
\end{itemize}

Indeed, assume for a moment that such sequences exist. By \eqref{eq:cupDj}, {\rm (2$_j$)}, and {\rm (5$_j$)}, there is a limit homotopy of holomorphic maps
\[ f_p^t =\lim_{j\to\infty} f_{p,j}^t\colon X\to A,\quad (p,t)\in P\times[0,1], \]
such that $\|f_p^t-f_{p,j}^t\|_{\overline D_j}<2\epsilon_j$ for all $(p,t)\in P\times[0,1]$ and all $j\geq 1$.  Thus, properties {\rm (6$_j$)} ensure that $f_p^t\in\O_*(X,Y)$ for all $(p,t)\in P\times[0,1]$; take into account \eqref{eq:cupDj}.  Moreover, conditions {\rm (1$_j$)} imply that $f_p^t=f(p)$ for all $(p,t)\in (P\times\{0\})\cup(Q\times[0,1])$, so the homotopy keeps $f|_Q$ fixed.  Finally, in view of {\rm (4$_j$)} we have $\int_\gamma f_p^1\theta=\int_\gamma \phi(p)$ for all loops $\gamma$ in $X$ and all points $p\in P$, that is, $\pi\circ f_p^1=\phi(p)$ for all $p\in P$.  Setting $f^1(p)=f_p^1$, $p\in P$, we have $\pi\circ f^1=\phi$ on $P$.  This completes the proof of the theorem under the assumption that the above-mentioned sequences exist. 

We shall construct the sequences $f_{p,j}^t$ and $\epsilon_j$, $j\geq 1$, satisfying {\rm (1$_j$)}--{\rm (6$_j$)} inductively, adapting the arguments of the proofs of \cite[Theorems 4.1 and 5.3]{Forstneric-Larusson-CAG}. 

The basis of the induction is given by the homotopy
\[ f_{p,1}^t = f(p)|_{\overline D_1},\quad (p,t)\in P\times[0,1]. \]
Choose $\epsilon_1>0$ small enough that {\rm (6$_1$)} holds. Since $Q\subset P$, property {\rm (1$_1$)} trivially holds.  Since $D_1$ is simply connected, we have $\int_\gamma f_{p,1}^1\theta=0=\int_\gamma \phi(p)$ for all loops $\gamma$ in $D_1$ and all points $p\in P$, so {\rm (4$_1$)} holds.  The other properties are vacuous for $j=1$.

For the inductive step, suppose that for some $j\geq 2$ we already have homotopies $f_{p,i}^t$ and numbers $\epsilon_i$ meeting the corresponding requirements for $i=1,\ldots,j-1$.  Let us prove the existence of suitable $f_{p,j}^t$ and $\epsilon_j$. 

We consider two cases, depending on whether $\rho$ has a critical value in $[c_{j-1},c_j]$ or not.  Recall that $\rho$ has at most one critical value in that interval.

\medskip

\noindent{\em Case 1: The noncritical case.}  Assume that $\rho$ has no critical value in $[c_{j-1},c_j]$, so $\overline D_{j-1}$ is a strong deformation retract of $\overline D_j$. Reasoning as in the proof of \cite[Theorem 4.1]{Forstneric-Larusson-CAG}, we embed the homotopy $f_{p,j-1}^t$ as the core 
\begin{equation}\label{eq:core}
	f_{p,j-1}^t=h_{p,0}^t
\end{equation}
of a period-dominating spray of holomorphic maps
\begin{equation}\label{eq:hpzt}
	h_{p,\zeta}^t\colon \overline D_{j-1}\to Y,\quad \zeta\in B,\; (p,t)\in P\times [0,1].
\end{equation}
Here, $B$ is a sufficiently small open ball centred at the origin in $\C^N$ for some $N\geq 1$, and $\zeta=(\zeta_1,\ldots,\zeta_N)$ is a parameter in $B$ on which the maps $h_{p,\zeta}^t$ depend holomorphically. 

In the proof of \cite[Theorem 4.1]{Forstneric-Larusson-CAG}, the period-domination follows from \cite[Lemma 5.1]{Alarcon-Forstneric-2014} (see also \cite[Lemma 3.6]{Alarcon-Forstneric-Crelle}) with respect to a fixed basis of the first homology group $H_1(D_{j-1}, \Z)$. In our case, we use exactly the same argument but apply the following generalisation of \cite[Lemma 5.1]{Alarcon-Forstneric-2014}.

\begin{lemma}   \label{lem:5.1}
Let $X$, $\theta$, and $Y$ be as in Theorem \ref{thm:main} (except that $Y$ need not be Oka).  Let $M\subset X$ be a smoothly bounded $\O(X)$-convex compact domain and let $C_1,\ldots,C_l$ be smooth embedded loops in $M$ that form a basis of $H_1(M,\Z)$ and only meet at a common point $p$ in $M$ and are otherwise mutually disjoint, such that the compact set $C=\bigcup\limits_{j=1}^l C_j\subset M$ is $\O(X)$-convex. 

Then for any holomorphic map $\psi\colon M\to Y$, there are an open neighbourhood $U$ of the origin in some $\C^N$ and a holomorphic map $ \Phi_\psi:U\times M\to Y$, such that $\Phi_\psi(0,\cdot)=\psi$ and the period map
\[ U \to (\C^n)^l, \quad \zeta\longmapsto \Big( \int_{C_j}\Phi_\psi(\zeta,\cdot)\theta\Big)_{j=1}^l, \]
has maximal rank equal to $ln$ at $\zeta=0$.
\end{lemma}

\begin{proof}
In the statement of \cite[Lemma 5.1]{Alarcon-Forstneric-2014}, it is assumed that the holomorphic $1$-form $\theta$ has no zeroes in $M$ and that $Y=A\setminus\{0\}$, where $A$ is an irreducible closed conical subvariety of $\C^n$ for some $n\ge 3$ that is not contained in any hyperplane of $\C^n$, such that $Y$ is smooth.  We point out the changes that are required in the proof given in \cite{Alarcon-Forstneric-2014} for it to work in our more general setting. 

First, we slightly modify the smooth embedded loops $C_1,\ldots, C_l$ in $M$ so that they avoid the zeroes of $\theta$.  This clearly is possible since the zero set of $\theta$ in $M$ is finite.  (Alternatively, instead of worrying about the choice of the loops, we can choose the points $x_{i,1},\ldots,x_{i,m}\in C_i\setminus\{p\}$ in the proof of \cite[Lemma 5.2]{Alarcon-Forstneric-2014}, $i=1,\ldots,l$, to lie in the complement in $C_i\setminus\{p\}$ of the zero set of $\theta$; again this is easy since the zero set of $\theta$ in $M$ is finite.)  As in the proof of \cite[Lemma 5.1]{Alarcon-Forstneric-2014}, Cartan's Theorem A applied on $A$ provides holomorphic tangent vector fields $V_1,\ldots,V_m$ on $Y$ that span $T_y Y$ for each $y\in Y$.  We denote by $\phi_t^j$ the flow of $V_j$ for small complex values of time $t$ and define, for a small open neighbourhood $U$ of the origin in $(\C^m)^l$, a smooth map
\[  \Psi\colon U\times C\to Y, \quad (\zeta,x)\longmapsto \phi_{\zeta_{1,1} h_{1,1}(x)}^1\circ\cdots\circ \phi_{\zeta_{l,m} h_{l,m}(x)}^m\big(\psi(x)\big), \]
where $\zeta=\big(\zeta_{j,1},\ldots,\zeta_{j,m}\big)_{j=1}^l\in U$, $h_{i,1},\ldots,h_{i,m}\colon C\to\C$ are smooth functions with support in $C_i$, $i=1,\ldots,l$, and $U$ is small enough that $\Psi$ assumes values in $Y$ (this requires compactness of $M$).  We complete the proof exactly as in \cite{Alarcon-Forstneric-2014},  keeping $U$ small enough that all the maps in the proof take values in $Y$.
\end{proof}

By the parametric Oka property with approximation (see \cite[Theorem 5.4.4]{Forstneric-2011}), we may now approximate the spray $h_{p,\zeta}^t$, uniformly on $\overline D_{j-1}$ and uniformly with respect to $p$, $t$, and $\zeta$, by a holomorphic spray of holomorphic maps
\[ g_{p,\zeta}^t\colon \overline D_j\to Y,\quad \zeta\in rB,\; (p,t)\in P\times [0,1], \]
for some number $0<r<1$ as close to $1$ as desired.  If the approximation is close enough, the implicit function theorem and the period-domination property of the spray $h_{p,\zeta}^t$ give a continuous map $\zeta\colon P\times[0,1]\to rB$, vanishing on $P\times\{0\}$ and on $Q\times[0,1]$, such that
\begin{equation}\label{eq:3j4j}
	\int_\gamma g_{p,\zeta(p,t)}^t\theta = \int_\gamma h_{p,0}^t\theta \overset{\textrm{(2)}}{=} \int_\gamma f_{p,j-1}^t\theta
\end{equation}
for all loops $\gamma$ in $D_{j-1}$ and all $(p,t)\in P\times[0,1]$.  Thus, the homotopy of holomorphic maps
\[f_{p,j}^t = g_{p,\zeta(p,t)}^t\colon \overline D_j\to Y,\quad (p,t)\in P\times[0,1], \]
satisfies conditions {\rm (3$_j$)} and {\rm (4$_j$)} in view of \eqref{eq:3j4j}, since we are assuming {\rm (3$_{j-1}$)} by the induction hypothesis and $\overline D_{j-1}$ is a strong deformation retract of $\overline D_j$.  Moreover, \eqref{eq:core}, {\rm (1$_{j-1}$)}, the fact that $\zeta(p,t)=0$ for all $(p,t)\in(P\times\{0\})\cup(Q\times[0,1])$, and the identity principle guarantee condition {\rm (1$_j$)}.  Furthermore, if the approximation of $h_{p,\zeta}^t$ by $g_{p,\zeta}^t$ is close enough, condition {\rm (2$_j$)} trivially holds and all the maps in the homotopy $f_{p,j}^t$ are nondegenerate since the maps in $f_{p,j-1}^t$ are.  Therefore, to complete the inductive step, it only remains to choose a number $\epsilon_j>0$ such that conditions {\rm (5$_j$)} and {\rm (6$_j$)} are satisfied.

\medskip

\noindent{\em Case 2: The critical case.}  Now assume that there is a critical value of $\rho$ in $(c_{j-1},c_j)$.  By our assumptions, there is exactly one critical value of $\rho$ in the interval.  This implies the existence of an embedded real analytic arc $E$ in $D_j\setminus D_{j-1}$, attached to $\overline D_{j-1}$ at both ends, meeting the boundary of $D_{j-1}$ transversely there, and otherwise disjoint from $\overline D_{j-1}$, such that the set
\[ S = \overline D_{j-1}\cup E\subset D_j \]
is a strong deformation retract of $\overline D_j$.  We choose the arc $E$ to contain no zeroes of $\theta$. 

We consider two cases, depending on whether the two endpoints of $E$ lie in the same connected component of $\overline D_{j-1}$ or not.

\smallskip

\noindent{\em Case 2.1.} Assume that the two endpoints of $E$ lie in the same connected component of $\overline D_{j-1}$.  Then there is an embedded closed real analytic curve $C$ in $S$ which contains $E$ and whose homology class belongs to $H_1(D_j, \Z)$ but not to $H_1(D_{j-1}, \Z)$.  We may assume that $C$ contains no zeroes of $\theta$. 

Observe that the set $S=\overline D_{j-1}\cup E=\overline D_{j-1}\cup C\subset X$ is $\O(X)$-convex.  We split $C$ into three subarcs $C=C_1\cup C_2\cup C_3$, lying end to end and being otherwise mutually disjoint, such that $C_3=C\cap \overline D_{j-1}$; hence $C_1\cup C_2=E$.  Also for $i=1,2$, we choose a closed subarc $C_i'$ of $C_i$ that lies in the relative interior of $C_i$.

Following the proof of \cite[Theorem 4.1]{Forstneric-Larusson-CAG}, we extend the homotopy $f_{p,j-1}^t\colon\overline D_{j-1}\to Y$, with the same name, to a continuous family of continuous maps $S=\overline D_{j-1}\cup C_1\cup C_2\to Y$ such that
\begin{equation}\label{eq:extension}
	f_{p,j-1}^t=f(p)|_S\quad \text{for all $(p,t)\in (P\times\{0\})\cup (Q\times[0,1])$}.
\end{equation}
Such an extension exists in view of {\rm (1$_{j-1}$)}.  

For the next step we wish to use \cite[Lemma 3.1]{Forstneric-Larusson-CAG}. It remains true with the punctured null quadric $\mathfrak A_*\subset\mathfrak A\subset\C^n$ replaced by the more general target $Y\subset A\subset \C^n$.  We shall indicate the necessary changes to the proof in \cite{Forstneric-Larusson-CAG}.  Since the convex hull ${\rm Co}(Y)$ of $Y$ is $\C^n$, in the notation of \cite{Forstneric-Larusson-CAG}, there is a number $r_1>0$ large enough that
\begin{equation}\label{eq:Omega}
	\sigma(P\times[0,1])\subset Y\cap r_1\B
	\quad\text{and}\quad
	\{\alpha_p^t\colon (p,t)\in P\times[0,1]\}\subset {\rm Co}(Y\cap r_1\B),
\end{equation}
where $\B$ denotes the open unit ball in $\C^n$.  Now we come to the main change to the proof of \cite[Lemma 3.1]{Forstneric-Larusson-CAG}.  Let $\delta>0$ be so small that there is a continuous function $a_\delta\colon Y\cap r_1\B\to (0,\delta]$ with the following properties.
\begin{itemize}
\item If for each $z\in Y\cap r_1\B$, we denote by $B_z$ the ball centered at $z$ with radius $a_\delta(z)$ in the complex affine subspace of dimension $n-\dim Y$ that passes through $z$ and is orthogonal to $Y$ at $z$, then $B_z\cap A=\{z\}$ for all $z\in Y\cap r_1\B$.  If $A=\C^n$, we set $B_z=\{z\}$.
\item  The balls $B_z$, $z\in Y\cap r_1\B$, are mutually disjoint.
\end{itemize}
Now let
\[ \Omega_\delta=\bigcup_{z\in Y\cap r_1\B} B_z.  \]
Since $Y\cap r_1\B\subset \Omega_\delta$, by \eqref{eq:Omega},
\[
	\sigma(P\times[0,1])\subset \Omega_\delta
	\quad\text{and}\quad
	\{\alpha_p^t\colon (p,t)\in P\times[0,1]\}\subset {\rm Co}(\Omega_\delta).
\] 
We consider the continuous retraction $\tau\colon \Omega_\delta\to Y\cap r_1\B$ given by
\[ \tau\vert_{B_z}=z,\quad z\in Y\cap r_1\B, \]
and assume that $\delta$ is so small that
\[ \left|\int_0^1 \gamma(s)\, ds-\int_0^1\tau(\gamma(s))\, ds\right|<\epsilon/4\quad \text{for all paths $\gamma\colon[0,1]\to\Omega_\delta$.}  \]
Here $\epsilon>0$ is the number given in the statement of \cite[Lemma 3.1]{Forstneric-Larusson-CAG}.  Using the domain $\Omega_\delta$ and the retraction $\tau$, the rest of the proof is exactly the same as the proof of \cite[Lemma 3.1]{Forstneric-Larusson-CAG} and we omit the details.  This concludes the proof of the generalisation of \cite[Lemma 3.1]{Forstneric-Larusson-CAG} to our setting.

We now continue the proof of the theorem. 

Given a number $\mu>0$, which will be specified later, \cite[Lemma 3.1]{Forstneric-Larusson-CAG} provides a continuous family of continuous maps
\[ g_p^{t,s}\colon C\to Y,\quad (p,t)\in P\times[0,1],\; s\in [0,1], \]
satisfying the following conditions.
\begin{itemize}
\item[\rm (i)] $g_p^{t,s}=f(p)|_C$ for all $(p,t)\in (P\times\{0\})\cup(Q\times[0,1])$ and all $s\in [0,1]$.
\item[\rm (ii)] $g_p^{t,0}=f_{p,j-1}^t|_C$ for all $(p,t)\in P\times[0,1]$.
\item[\rm (iii)] $g_p^{t,s}=f_{p,j-1}^t$ on $C\setminus C_1'$ for all $(p,t)\in P\times[0,1]$ and all $s\in[0,1]$.
\item[\rm (iv)] $\big|\int_C g_p^{t,1}\theta-\int_C \phi(p)\big|<\mu$ for all $p\in P$.
\end{itemize}
Since $\theta$ has no zeroes on $C$, $\theta$ has no zeroes in a small open neighbourhood of $C$.  This enables us to apply the generalised \cite[Lemma 3.1]{Forstneric-Larusson-CAG} to obtain the maps $g_p^{t,s}$ exactly as in the proof of \cite[Theorem 4.1]{Forstneric-Larusson-CAG}, but replacing the open Riemann surface $X$ (called $M$ in \cite{Forstneric-Larusson-CAG}) by an open neighbourhood of $C$ in $X$ in which $\theta$ vanishes nowhere. 

Next, using Lemma \ref{lem:5.1}, we embed $f_{p,j-1}^t|_{C_2}$ as the core
\begin{equation}\label{eq:corecritical}
	f_{p,j-1}^t|_{C_2}=h_{p,0}^t
\end{equation}
of a period-dominating spray of continuous maps
\[ h_{p,\zeta}^t\colon C_2\to Y,\quad (p,t)\in P\times[0,1],\; \zeta\in B, \]
where, as in the noncritical case, $B$ is an open ball containing the origin in $\C^N$ for some $N\geq 1$ and $\zeta$ is a parameter in $B$, such that
\begin{equation}\label{eq:C2C2'}
	h_{p,\zeta}^t=f_{p,j-1}^t\quad \text{on $C_2\setminus C_2'$ for all $(p,t)\in P\times[0,1]$}.
\end{equation}
When applying Lemma \ref{lem:5.1}, we use the fact that $\theta$ has no zeroes on $C$ and hence none on $C_2$.

Assuming that $\mu>0$ is small enough, condition (iv), the implicit function theorem, and the period-domination property of the spray $h_{p,\zeta}^t$ give a continuous map $\zeta\colon P\times[0,1]\to B$, vanishing on $P\times\{0\}$ and on $Q\times[0,1]$, such that the family of maps
\[ u_p^t\colon S\to Y,\quad (p,t)\in P\times[0,1], \]
given by
\[
	u_p^t|_{\overline D_{j-1}}  = f_{p,j-1}^t,\quad
	u_p^t|_{C_1}  = g_p^{t,1}|_{C_1},  \quad
	u_p^t|_{C_2}  =  h_{p,\zeta(p,t)}^t \quad\text{for all $(p,t)\in P\times[0,1]$}
\]
satisfies the following properties.
\begin{itemize}
\item[\rm (a)] $u_p^t$ is a continuous map for all $(p,t)\in P\times[0,1]$.
\item[\rm (b)] $u_p^t=f(p)|_S$ for all $(p,t)\in (P\times\{0\})\cup(Q\times[0,1])$.
\item[\rm (c)] $\int_\gamma u_p^t\theta=\int_\gamma \phi(p)$ for all loops $\gamma$ in $S$ and all $p\in P$. 
\end{itemize}
Indeed, condition {\rm (a)} follows from {\rm (iii)} and \eqref{eq:C2C2'}, while {\rm (b)} follows from {\rm (i)}, \eqref{eq:corecritical}, and the fact that $\zeta(p,t)=0$ for all $(p,t)\in (P\times\{0\})\cup(Q\times[0,1])$. To ensure {\rm (c)}, we use {\rm (4$_{j-1}$)} and {\rm (iv)}, and exploit the period-domination property of $h_{p,\zeta}^t$, assuming that $\mu>0$ has been chosen sufficiently small.

To finish the proof we apply the same argument as in the noncritical case but replacing $f_{p,j-1}^t$ by the homotopy  $u_p^t$.  This is possible in view of conditions {\rm (a)}, {\rm (b)}, and {\rm (c)}, and the facts that $u_p^t|_{\overline D_{j-1}}  = f_{p,j-1}^t$ and $S$ is $\O(X)$-convex.

\smallskip

\noindent{\em Case 2.2.}   Now assume that the endpoints of $E$ lie in different connected components of $\overline D_{j-1}$.  Then $E$ does not close to an embedded loop in $\overline D_{j-1}\cup E$ and hence no new element of the homology basis appears.  In this case, just extending $f_{p,j-1}^t\colon\overline D_{j-1}\to Y$ to a continuous family of continuous maps $S=\overline D_{j-1}\cup E\to Y$ satisfying \eqref{eq:extension} enables us to apply the argument that we used in the noncritical case.

This completes the inductive step and concludes the proof of the theorem.
\end{proof}

\end{document}